\documentclass{amsart}
\pdfoutput=1
\usepackage{amssymb}
\usepackage{url}
\usepackage{enumerate}
\usepackage{thmtools}
\usepackage[numbers]{natbib}
\usepackage{hyperref} % NOTE: hyperref should generally be the last package loaded

\hypersetup{
  pdftitle={Monotone Hurwitz numbers in genus zero},
  pdfauthor={I. P. Goulden, Mathieu Guay-Paquet, and Jonathan Novak},
}

\declaretheorem[numberwithin=section]{theorem}

\declaretheorem[numberlike=theorem]{corollary}

\declaretheorem[numberlike=theorem, style=definition]{definition}

\declaretheorem[numberwithin=section, style=remark]{remark}
\numberwithin{equation}{section}

\author{I.~P.~Goulden}
\address{Department of Combinatorics \& Optimization\\University of Waterloo\\Canada}
\email{ipgoulden@uwaterloo.ca}

\author{Mathieu Guay-Paquet}
\address{Department of Combinatorics \& Optimization\\University of Waterloo\\Canada}
\email{mguaypaq@uwaterloo.ca}
\thanks{IPG and MGP were supported by NSERC}

\author{Jonathan Novak}
\address{Department of Mathematics. Massachusetts Institute of Technology, USA}
\email{jnovak@math.mit.edu}

\title{Monotone Hurwitz numbers in genus zero}

\date{\today}

\newcommand{\QQ}{\mathbb{Q}}

\newcommand{\id}{\mathrm{id}}

\newcommand{\pp}{\mathbf{p}}
\newcommand{\qq}{\mathbf{q}}
\newcommand{\xx}{\mathbf{x}}
\newcommand{\yy}{\mathbf{y}}

\newcommand{\hur}{H}
\newcommand{\Hur}{\mathbf{H}}
\newcommand{\mon}{\vec{H}}
\newcommand{\Mon}{\vec{\mathbf{H}}}
\newcommand{\C}{C}
\newcommand{\guess}{\mathbf{F}}

\newcommand{\D}{\mathrm{D}}
\newcommand{\E}{\mathrm{E}}
\newcommand{\DD}{\mathcal{D}}
\newcommand{\EE}{\mathcal{E}}

\newcommand{\abs}[1]{\left|{#1}\right|}
\newcommand{\diff}[2][]{\frac{\partial{#1}}{\partial{#2}}}
\newcommand{\sdiff}[3][]{\frac{\partial^2{#1}}{\partial{#2}\partial{#3}}}

\DeclareMathOperator{\ch}{ch}
\DeclareMathOperator{\Aut}{Aut}
\DeclareMathOperator*{\Split}{Split}
\DeclareMathOperator{\trace}{tr}

\begin{document}

\begin{abstract}
  Hurwitz numbers count branched covers of the Riemann sphere with specified ramification data, or equivalently, transitive permutation factorizations in the symmetric group with specified cycle types. Monotone Hurwitz numbers count a restricted subset of the branched covers counted by the Hurwitz numbers, and have arisen in recent work on the the asymptotic expansion of the Harish-Chandra-Itzykson-Zuber integral. In this paper we begin a detailed study of monotone Hurwitz numbers. We prove two results that are reminiscent of those for classical Hurwitz numbers. The first is the monotone join-cut equation, a partial differential equation with initial conditions that characterizes the generating function for monotone Hurwitz numbers in arbitrary genus. The second is our main result, in which we give an explicit formula for monotone Hurwitz numbers in genus zero.
\end{abstract}

\maketitle

\setcounter{tocdepth}{2}
\tableofcontents

%----------------------------------------------------------------
\section{Introduction}
%----------------------------------------------------------------

Hurwitz numbers count branched covers of the Riemann sphere with specified ramification data. They have been the subject of much mathematical interest in recent years, especially through the ELSV formula, given in~\cite{ekedahl-lando-shapiro-vainshtein01}, which expresses a Hurwitz number as a Hodge integral over the moduli space of stable curves of a given genus with a given number of marked points. This has led to a number of new proofs (see, \textit{e.g.}, Okounkov and Pandharipande~\cite{okounkov-pandharipande09} and Kazarian and Lando~\cite{kazarian-lando07}) of Witten's conjecture~\cite{witten91} (first proved by Kontsevich~\cite{kontsevich92}), which states that a particular generating function for intersection numbers satisfies the KdV hierarchy of partial differential equations. The interest in Hurwitz numbers has much to do with these rich connections that their study has revealed between mathematical physics and algebraic geometry. There is also a connection with algebraic combinatorics because of the bijection, due to Hurwitz~\cite{hurwitz91}, between branched covers of the sphere and transitive factorizations in the symmetric group (see, \textit{e.g.}, Goulden, Jackson and Vainshtein~\cite{goulden-jackson-vainshtein00}).

Monotone Hurwitz numbers, introduced in~\cite{goulden-guay-paquet-novak11a}, count a restricted subset of the branched covers counted by the Hurwitz numbers. The topic of~\cite{goulden-guay-paquet-novak11a} is the Harish-Chandra-Itzykson-Zuber (HCIZ) integral (see, \textit{e.g.}, \cite{harish-chandra57},~\cite{itzykson-zuber80},~\cite{zinn-justin-zuber03})
\[
  I_N(z, A_N, B_N) = \int\limits_{\mathbf{U}(N)} e^{zN\trace(A_N U B_N U^*)} \mathrm{d}U.
\]
Here the integral is over the group of $N \times N$ complex unitary matrices against the normalized Haar measure, $z$ is a complex parameter, and $A_N,B_N$ are $N\times N$ complex matrices. Since $\mathbf{U(N)}$ is compact, $I_N$ is an entire function of $z \in \mathbb{C}$. Consequently, the function
\[
  F_N(z, A_N, B_N) = N^{-2} \oint\limits_0^z \frac{I_N'(\zeta, A_N, B_N)}{I_N(\zeta, A_N, B_N)} \mathrm{d}\zeta
\]
is well-defined and holomorphic in a neigbourhood of $z = 0$, and satisfies
\[
	I_N(z, A_N, B_N) = e^{N^2 F_N(z, A_N, B_N)}
\]
on its domain of definition. In~\cite{goulden-guay-paquet-novak11a}, we proved that, for two specified sequences of normal matrices $A=(A_N)_{N=1}^{\infty}$, $B=(B_N)_{N=1}^{\infty}$ which grow in a sufficiently regular fashion, the derivatives $F_N^{(d)}(0, A_N, B_N)$ of $F_N$ at $z = 0$ admit an $N \rightarrow \infty$ asymptotic expansion on the scale $N^{-2}$ whose $g$th coefficient is a generating function for the monotone double Hurwitz numbers of degree $d$ and genus $g$. This is analogous to the well-known genus expansion of Hermitian matrix models, whose coefficients are generating functions enumerating graphs on surfaces (see \textit{e.g.}~\cite{zvonkin97}). In this paper, we begin a detailed study of monotone Hurwitz numbers.

%----------------------------------------------------------------
\subsection{Hurwitz numbers}
%----------------------------------------------------------------

The \emph{single Hurwitz numbers} count $d$-sheeted branched covers of the Riemann sphere by a Riemann surface where we allow arbitrary, but fixed, branching at one ramification point and only simple branching at other ramification points. Using the Hurwitz~\cite{hurwitz91} encoding of a branched cover as a factorization in the symmetric group, we obtain the following identification with an enumeration question in the symmetric group: given a partition $\alpha \vdash d$ and an integer $r \geq 0$, the single Hurwitz number $\hur^r(\alpha)$ is the number of factorizations
\begin{equation}\label{eq:single}
  (a_1 \, b_1) (a_2 \, b_2) \cdots (a_r \, b_r) = \sigma
\end{equation}
in the symmetric group $S_d$, where
\begin{itemize}
  \item
    $(a_1 \, b_1), (a_2 \, b_2), \ldots, (a_r \, b_r)$ are transpositions,
  \item
    $\sigma$ is in the conjugacy class $\C_\alpha$ of permutations with cycle type $\alpha$, and
  \item
    the subgroup $\langle (a_1 \, b_1), (a_2 \, b_2), \ldots, (a_r \, b_r) \rangle$ acts transitively on the ground set $\{1, 2, \ldots, d\}$.
\end{itemize}
Each factorization corresponds to a branched cover of the Riemann sphere, and by the Riemann-Hurwitz formula, the genus $g$ of the cover is given by the relation
\begin{equation}\label{eq:riemann-hurwitz}
 r = d + \ell(\alpha) + 2g - 2,
\end{equation}
where $\ell(\alpha)$ denotes the number of parts of $\alpha$. Depending on the context, we will write $\hur_g(\alpha)$ interchangeably with $\hur^r(\alpha)$, using the convention that \eqref{eq:riemann-hurwitz} always holds.

\begin{remark}
Equation~\eqref{eq:single}, when rewritten as $(a_1 \, b_1) (a_2 \, b_2) \cdots (a_r \, b_r) \sigma^{-1} = \mathrm{id}$, translates into a \emph{monodromy} condition for the corresponding cover, in which $\sigma^{-1}$ specifies the branching for the point with arbitrary ramification, and $(a_i \, b_i)$, $i = 1, \ldots, r$ specifies the (simple) branching at the remaining ramification points. The transitivity condition for the factorization translates to the requirement that the corresponding cover is \emph{connected}.
\end{remark}

%----------------------------------------------------------------
\subsection{Monotone Hurwitz numbers}
%----------------------------------------------------------------

The \emph{monotone single Hurwitz number} $\mon^r(\alpha)$ is the number of factorizations~\eqref{eq:single} counted by the single Hurwitz number~$\hur^r(\alpha)$, but with the additional restriction that
\begin{equation}\label{eq:monotone}
  b_1 \leq b_2 \leq \cdots \leq b_r,
\end{equation}
where $a_i < b_i$ by convention. As with Hurwitz numbers, depending on the context, we will write $\mon_g(\alpha)$ interchangeably with $\mon^r(\alpha)$, with the understanding that \eqref{eq:riemann-hurwitz} holds. We will refer to a factorization~\eqref{eq:single} with restriction~\eqref{eq:monotone} as a \emph{monotone} factorization of $\sigma$.

\begin{remark}
In Hurwitz's encoding, the ground set $\{1, 2, \ldots, d\}$ corresponds to the set of sheets of the branched cover, once branch cuts have been chosen and the sheets have been labelled. In the case of Hurwitz numbers, the labelling of the sheets is immaterial, so Hurwitz numbers are usually defined to count branched covers with \emph{unlabelled} sheets, which differs from our definition above by a factor of $d\,!$. However, for monotone Hurwitz numbers, the monotonicity condition depends on a total ordering of the sheets, so the labelling does matter in this case. Thus, for consistency, our convention is that both kinds of Hurwitz numbers count branched covers with labelled sheets.
\end{remark}

%----------------------------------------------------------------
\subsection{Main result}
%----------------------------------------------------------------

In \autoref{sec:solution}, we obtain the following theorem. This is our main result, and gives an explicit formula for the genus zero monotone Hurwitz numbers.

\begin{theorem}\label{thm:gzformula}
  The genus zero monotone single Hurwitz number $\mon_0(\alpha)$, $\alpha \vdash d$ is given by
  \begin{equation}\label{eq:mon0}
    \mon_0(\alpha) = \frac{d\,!}{\abs{\Aut \alpha}} (2d + 1)^{\overline{\ell(\alpha)-3}} \; \prod_{j=1}^{\ell(\alpha)} \binom{2\alpha_j}{\alpha_j},
  \end{equation}
  where
  \[
    (2d + 1)^{\overline{k}} = (2d + 1) (2d + 2) \cdots (2d + k)
  \]
  denotes a rising product with $k$ factors, and by convention
  \[
    (2d + 1)^{\overline{k}} = \frac{1}{(2d + k + 1)^{\overline{-k}}}
  \]
  for $k<0$.
\end{theorem}

In the special case that $\alpha = (d)$, the partition with a single part equal to $d$, \autoref{thm:gzformula} becomes
\[
  \mon_0((d)) = \frac{(2d-2)!}{d\,!} = (d-1)! \, C_{d-1},
\]
where $C_{d-1} = \frac{1}{d}\binom{2d-2}{d-1}$ is a \emph{Catalan} number. This case was previously obtained by Gewurz and Merola~\cite{gewurz-merola06}, who used the term \emph{primitive} for these factorizations.
In the special case that $\alpha = (1^d)$, the partition with all parts equal to 1, \autoref{thm:gzformula} becomes
\[
  \mon_0((1^d)) = (d-1)! \, 2^{d-1}.
\]
Via the connection between monotone Hurwitz numbers and the HCIZ integral established in~\cite{goulden-guay-paquet-novak11a}, this case is equivalent to a result previously obtained by Zinn-Justin~\cite{zinn-justin02} for the HCIZ integral.

\autoref{thm:gzformula} is strikingly similar to the well-known explicit formula for the genus zero Hurwitz numbers
\begin{equation}\label{eq:hur0}
  \hur_0(\alpha) = \frac{d\,!}{\abs{\Aut \alpha}} (d+\ell(\alpha)-2)! \, d^{\,\ell(\alpha)-3} \; \prod_{j=1}^{\ell(\alpha)} \frac{\alpha_j^{\alpha_j}}{\alpha_j!},
\end{equation}
published without proof by Hurwitz~\cite{hurwitz91} in 1891 (see also Strehl~\cite{strehl96}) and independently rediscovered and proved a century later by Goulden and Jackson~\cite{goulden-jackson97}.

There is another case in which an explicit formula similar to \autoref{thm:gzformula} is known. This is the case where we allow arbitrary, but fixed, branching at a specified ramification point and arbitrary branching at all other ramification points, which has been studied by Bousquet-M\'elou and Schaeffer~\cite{bousquet-melou-schaeffer00}. Given a partition $\alpha \vdash d$ and integers $r,g \geq 0$, let $G_{\!g}^r(\alpha)$ be the number of factorizations
\begin{equation}\label{eq:arbitrary}
  \rho_1 \rho_2 \cdots \rho_r = \sigma
\end{equation}
in the symmetric group $S_d$ which satisfy the conditions
\begin{itemize}
  \item
    $\rho_1, \rho_2, \ldots, \rho_r \in S_d$,
  \item
    $\sigma \in \C_\alpha$,
  \item
    $\langle \rho_1, \rho_2, \ldots, \rho_r \rangle$ acts transitively on $\{1, 2, \ldots, d\}$,
\end{itemize}
and
\begin{equation}\label{eq:riemann-hurwitzBMS}
  \sum_{j=1}^r \operatorname{rank}(\rho_j) = d + \ell(\alpha) + 2g - 2,
\end{equation}
where $\operatorname{rank}(\rho_j)$ is $d$ minus the number of cycles of $\rho_j$. Each such factorization corresponds to a branched cover of the Riemann sphere, and~\eqref{eq:riemann-hurwitzBMS}, by the Riemann-Hurwitz formula, specifies the genus $g$ of the cover.

\begin{remark}
Note that in this case there is more freedom for the parameters $\alpha$, $r$, $g$ than for the Hurwitz and monotone Hurwitz cases above. In particular, given $\alpha$ and $r$, the choice for $g$ is \emph{not} unique in~\eqref{eq:riemann-hurwitzBMS} above. This explains why we have used both parameters $r$ and $g$ in the notation $G_{\!g}^r(\alpha)$.
\end{remark}

Bousquet-M\'elou and Schaeffer~\cite{bousquet-melou-schaeffer00} solved this problem in full generality for genus zero, using a bijective correspondence to constellations and thence to a family of bicoloured trees. (For more on \emph{constellations}, see Lando and Zvonkin~\cite{lando-zvonkin04}.) They proved that
\begin{equation}\label{eq:BMS0}
  G_{\!0}^r(\alpha) = \frac{d\,!}{\abs{\Aut \alpha}} \, r \, ((r-1)d-\ell(\alpha)+2)^{\overline{\ell(\alpha)-2}} \; \prod_{j=1}^{\ell(\alpha)} \binom{r\alpha_j-1}{\alpha_j},
\end{equation}
an explicit form that is again strikingly similar to both~\eqref{eq:mon0} and~\eqref{eq:hur0}.

The explicit formulas~\eqref{eq:mon0},~\eqref{eq:hur0} and~\eqref{eq:BMS0} feature remarkably simple combinatorial functions, but we know of no uniform bijective method to explain these formulas.

%----------------------------------------------------------------
\subsection{Join-cut equations}
%----------------------------------------------------------------

The proof that we give for \autoref{thm:gzformula}, our main result, involves the generating function for monotone single Hurwitz numbers
\begin{equation}\label{eq:defMon}
  \Mon(z, t, \pp) = \sum_{d\geq 1} \frac{z^d}{d\,!} \sum_{r\geq 0} t^r \sum_{\alpha \vdash d} \mon^r(\alpha) p_\alpha,
\end{equation}
which is a formal power series in the indeterminates $z, t$ and the countable set of indeterminates $\pp = (p_1, p_2, \ldots)$, and where $p_\alpha$ denotes the product $\prod_{j=1}^{\ell(\alpha)} p_{\alpha_j}$. From a combinatorial point of view,
\begin{itemize}
  \item
    $z$ is an exponential marker for the size $d$ of the ground set,
  \item
    $t$ is an ordinary marker for the number $r$ of transpositions, and
  \item
    $p_1, p_2, \ldots$ are ordinary markers for the cycle lengths of $\sigma$.
\end{itemize}
In \autoref{sec:joincut}, we obtain the following theorem, which gives a partial differential equation with initial condition that uniquely specifies the generating function $\Mon$. The proof that we give is a combinatorial join-cut analysis, and we refer to the partial differential equation in \autoref{thm:joincut} as the \emph{monotone join-cut equation}.

\begin{restatable}{theorem}{restatejoincut}\label{thm:joincut}
  The generating function $\Mon$ is the unique formal power series solution of the partial differential equation
  \[
    \frac{1}{2t}\left( z\diff[\Mon]{z} - z p_1 \right) = \frac{1}{2} \sum_{i,j \geq 1} \left( (i+j)p_i p_j \diff[\Mon]{p_{i+j}} + ij p_{i+j} \sdiff[\Mon]{p_i}{p_j} + ij p_{i+j} \diff[\Mon]{p_i} \diff[\Mon]{p_j} \right)
  \]
  with the initial condition $[z^0] \Mon = 0$.
\end{restatable}

\autoref{thm:joincut} is again strikingly similar to the situation for the classical single Hurwitz numbers. To make this precise, consider the generating function for the classical single Hurwitz numbers
\begin{equation}
  \Hur(z, t, \pp) = \sum_{d\geq 1}\frac{z^d}{d\,!} \sum_{r\geq 0} \frac{t^r}{r!} \sum_{\alpha \vdash d} \hur^r(\alpha) p_\alpha.
\end{equation}
It is well-known (see~\cite{goulden-jackson97,goulden-jackson-vainshtein00}) that $\Hur$ is the unique formal power series solution of the partial differential equation
\begin{equation}\label{eq:classicaljoincut}
  \diff[\Hur]{t} = \frac{1}{2} \sum_{i,j \geq 1} \left( (i+j)p_i p_j \diff[\Hur]{p_{i+j}} + ij p_{i+j} \sdiff[\Hur]{p_i}{p_j} + ij p_{i+j} \diff[\Hur]{p_i} \diff[\Hur]{p_j}\right)
\end{equation}
with the initial condition $[t^0] \Hur = z p_1$. Equation~\eqref{eq:classicaljoincut} is called the (classical) \emph{join-cut equation}, and has exactly the same differential forms on the right-hand side as the monotone join-cut equation given in \autoref{thm:joincut}. There are, however significant differences on the left-hand side between these two versions of the join-cut equation. In the classical case~\eqref{eq:classicaljoincut}, the left-hand side is a first derivative in $t$; in the monotone case~(\autoref{thm:joincut}), the left-hand side is a first \emph{divided difference} in $t$, and also involves differentiation in $z$.

\begin{remark}
  The difference in the left-hand sides between the two join-cut equations is related to the fact that the generating functions $\Hur$ and $\Mon$ differ in the combinatorial role played by the indeterminate $t$. In the case of $\Hur$, $t$ is an \emph{exponential} marker for the number $r$ of transpositions, while in the case of $\Mon$, $t$ is an \emph{ordinary} marker. This difference is for technical combinatorial reasons, explained in \autoref{sec:group-algebra}.
\end{remark}

Our proof of the explicit formula in \autoref{thm:gzformula} for genus zero monotone Hurwiz numbers proceeds by verification using a variant of the monotone join-cut equation given in \autoref{thm:joincut}. In general terms, this is how we obtained the explicit formula~\eqref{eq:mon0} for genus zero classical Hurwitz numbers in~\cite{goulden-jackson97}, using the classical join-cut equation~\eqref{eq:classicaljoincut}. However, the technical details of this verification are quite different in this paper because of the change in the left-hand side between these two different versions of the join-cut equation.

%----------------------------------------------------------------
\subsection{The group algebra of the symmetric group}\label{sec:group-algebra}
%----------------------------------------------------------------

In this section we express the generating functions $\Hur$ and $\Mon$ in terms of elements of the centre of the group algebra $\QQ[S_d]$. This is not an essential part of our proof of the main result, but will help explain why the indeterminate $t$ for transposition factors is an exponential marker in $\Hur$, whereas it is an ordinary marker in $\Mon$. In addition, it gives a convenient proof of the fact that the number of monotone factorizations of a permutation $\sigma$ depends only on its conjugacy class, which we will need in \autoref{sec:joincut} to prove the monotone join-cut equation.

First we consider the single Hurwitz numbers. For any $\alpha \vdash d$, let $\mathsf{C}_\alpha$ be the formal sum of all elements of the conjugacy class $\C_\alpha$, which consists of all permutations of cycle type $\alpha$, considered as an element of $\QQ[S_d]$. It is well-known that the centre of $\QQ[S_d]$ consists precisely of linear combinations of the $\mathsf{C}_\alpha$. If we drop the transitivity condition for single Hurwitz numbers, the generating function for the resulting not-necessarily-transitive factorizations becomes
\begin{align*}
  \tau(z, t, \pp)
    &= \sum_{d \geq 0} \frac{z^d}{d\,!} \sum_{r \geq 0} \frac{t^r}{r!} \sum_{\alpha \vdash d} p_\alpha \sum_{\sigma \in \C_\alpha} [\sigma] \mathsf{C}_{2,1^{d-2}}^r \\
    &= \sum_{d \geq 0} \frac{z^d}{d\,!} \sum_{r \geq 0} \frac{t^r}{r!} \sum_{\alpha \vdash d} p_\alpha \abs{\C_{\alpha}} [\mathsf{C}_{\alpha}] \mathsf{C}_{2,1^{d-2}}^r,
\end{align*}
where we have used the notation $[A]B$ for the \emph{coefficient} of $A$ in the expansion of $B$. The constant term 1 corresponding to $d = 0$ has been added to the generating function $\tau$ for combinatorial reasons, described as follows. When we drop the transitivity condition for single Hurwitz numbers, each resulting not-necessarily-transitive factorization can be split into disjoint transitive factorizations by restricting it to the orbits of the group $\langle (a_1 \, b_1), \ldots, (a_r \, b_r) \rangle$ on the ground set. Each of these orbits is a subset of the ground set $\{1,\ldots, d\}$, and the set of transpositions that act on pairs of elements in a given orbit is a subset of the positions $\{1, \ldots, r\}$ in the factorization. Conversely, transitive factorizations on disjoint ground sets can be combined by shuffling their transpositions in any way that preserves the order of transpositions acting on the same component of the ground set. Thus, each factorization counted by $\tau$ is an unordered collection of the transitive factorizations counted by $\Hur$, in which the variables $z$ (marking $d$) and $t$ (marking $r$) are both \emph{exponential}. From the Exponential Formula for exponential generating functions (see, \textit{e.g.},~\cite{goulden-jackson04}), this situation is captured by the equation
\[
  \Hur(z, t, \pp) = \log\tau(z, t, \pp).
\]

\begin{remark}
  The coefficient of $z^d t^r / r!$ in $\tau(z, t, \pp)$ is in fact the image of $\mathsf{C}_{2,1^{d-2}}^r$ under the characteristic map $\ch$ of Macdonald~\cite[p.~113]{macdonald95}, if one interprets the indeterminates $p_1, p_2, \dots$ as power sum symmetric functions. This can be expressed in the basis of Schur symmetric functions using irreducible characters of $S_d$, and then the tools of representation theory become available. While this is an interesting approach, we will not be using it here.
\end{remark}

Now we turn to monotone single Hurwitz numbers. While the monotonicity condition may seem artificial, it arises naturally in the group algebra $\QQ[S_d]$ via the Jucys-Murphy elements $\mathsf{J}_i$, defined by
\[
 \mathsf{J}_i = (1 \, i) + (2 \, i) + \cdots + (i-1 \, i), \qquad i = 1, \ldots, d.
\]
If we drop the transitivity condition for monotone single Hurwitz numbers, the generating function for the resulting not-necessarily-transitive factorizations becomes
\begin{align*}
  \vec{\tau}(z, t, \pp)
    &= \sum_{d \geq 0} \frac{z^d}{d\,!} \sum_{r \geq 0} t^r \sum_{\alpha \vdash d} p_\alpha \sum_{\sigma \in \C_\alpha}  [\sigma] \sum_{1 \leq b_1 \leq \cdots \leq b_r \leq d} \mathsf{J}_{b_1} \cdots \mathsf{J}_{b_r} \\
    &= \sum_{d \geq 0} \frac{z^d}{d\,!} \sum_{r \geq 0} t^r \sum_{\alpha \vdash d} p_\alpha \sum_{\sigma \in \C_\alpha}  [\sigma] h_r(\mathsf{J}_1, \ldots, \mathsf{J}_d),
\end{align*}
where $h_r$ is the $r$th complete symmetric polynomial. Jucys~\cite{jucys74} showed that the set of symmetric polynomials in the Jucys-Murphy elements is exactly the centre of $\QQ[S_d]$, so we obtain immediately that
\[
  \vec{\tau}(z, t, \pp) = \sum_{d \geq 0} \frac{z^d}{d\,!} \sum_{r \geq 0} t^r \sum_{\alpha \vdash d} p_\alpha \abs{\C_\alpha} [\mathsf{C}_\alpha] h_r(\mathsf{J}_1, \ldots, \mathsf{J}_d).
\]
This time, when we drop the transitivity condition for single monotone Hurwitz numbers, each resulting not-necessarily-transitive factorization can again be split into disjoint transitive factorizations by restricting it to the orbits of the group $\langle (a_1 \, b_1), \ldots, (a_r \, b_r) \rangle$ on the ground set. Each of these orbits is a subset of the ground set $\{1, \ldots, d\}$, and the set of transpositions that act on pairs of elements in a given orbit is a subset of the positions $\{1, \ldots, r\}$ in the factorization. However, this time, to preserve monotonicity, transitive factorizations on disjoint ground sets can be combined by shuffling their transpositions in only one way. Thus, each factorization counted by $\tau$ is an unordered collection of the transitive factorizations counted by $\Hur$, in which only the variable $z$ (marking $d$) is \emph{exponential}. The Exponential Formula for exponential generating functions then gives
\[
  \Mon(z, t, \pp) = \log\vec{\tau}(z, t, \pp).
\]

\begin{remark}
  As with $\tau(z, t, \pp)$, the coefficients of $\vec{\tau}(z, t, \pp)$ can be expressed in terms of the characteristic map $\ch$ of Macdonald~\cite{macdonald95} to provide a link with representation theory. This is particularly interesting in view of Okounkov and Veshik's approach to the representation theory of the symmetric group~\cite{okounkov-vershik96}, which features the Jucys-Murphy elements prominently. However, we will not be exploring this connection further in this paper.
\end{remark}

\begin{remark}
  Related results on complete symmetric functions of the Jucys-Murphy elements have been obtained by Lassalle~\cite{lassalle10} and F\'eray~\cite{feray11}. The recurrences they obtain seem to be of a completely different nature than those we obtain in this paper.
\end{remark}

%----------------------------------------------------------------
\subsection{Outline of paper}
%----------------------------------------------------------------

In \autoref{sec:joincut}, we prove the monotone join-cut equation of \autoref{thm:joincut}. This is based on a combinatorial join-cut analyis for monotone Hurwitz numbers that appears in \autoref{sec:recurrence}. The monotone join-cut equation itself is then deduced in \autoref{sec:jcequation}. The join-cut analysis also yields another system of equations that appears in \autoref{sec:topological}, and is referred to there as a topological recursion.

In \autoref{sec:intermed}, we recast the monotone join-cut equation, breaking it up into a separate join-cut equation for each genus, and expressing these equations in an algebraic form that is more convenient to solve. This is based on some algebraic operators that are introduced in \autoref{sec:liftprojsplit}, and the transformed system of equations appears in \autoref{sec:recasting}.

In \autoref{sec:solution}, we prove the main result of this paper, \autoref{thm:gzformula}, which gives an explicit formula for monotone Hurwitz numbers in genus zero. Our method is to repackage the formula as a generating function $\guess$, and to show that it satisfies the genus zero monotone join-cut equation which characterizes the generating function $\Mon_0$ for genus zero monotone Hurwitz numbers.
In \autoref{sec:lagrange}, we introduce transformed variables and use Lagrange's Implicit Function Theorem to give a closed form for a differential form applied to the series $\guess$. 
In \autoref{sec:invert}, we invert this differential form.
In \autoref{sec:genuszero}, we describe the action of the algebraic operators of \autoref{sec:intermed} on the transformed variables and deduce the main result.

%----------------------------------------------------------------
\section{Join-cut analysis}\label{sec:joincut}
%----------------------------------------------------------------

In this section, we analyze the effect of removing the last factor in a transitive monotone factorization, which leads to a recurrence relation for monotone single Hurwitz numbers. From this, we obtain the monotone join-cut equation of \autoref{thm:joincut}, which uniquely characterizes the generating function $\Mon(z, t, \pp)$, and a system of equations for a different generating function that we refer to as a \emph{topological recursion}.

%----------------------------------------------------------------
\subsection{Recurrence relation}\label{sec:recurrence}
%----------------------------------------------------------------

For a partition $\alpha \vdash d$, let $M^r(\alpha)$ be the number of transitive monotone factorizations of a fixed but arbitrary permutation $\sigma \in S_d$ of cycle type $\alpha$ into $r$ transpositions. By the discussion in \autoref{sec:group-algebra}, this number only depends on the cycle type of $\sigma$, so it is well-defined, and we immediately have
\begin{equation}\label{eq:defM}
  \mon^r(\alpha) = \abs{C_\alpha} M^r(\alpha).
\end{equation}

\begin{theorem}\label{thm:recur}
  The numbers $M^r(\alpha)$ are uniquely determined by the initial condition
  \[
    M^0(\alpha) = \begin{cases}
      1 &\text{if $\alpha = (1)$}, \\
      0 &\text{otherwise},
    \end{cases}
  \]
  and the recurrence
  \begin{multline}\label{eq:recurrence}
    M^{r+1}(\alpha \cup \{k\})
      = \sum_{k' \geq 1} k' m_{k'}(\alpha) M^r(\alpha \setminus \{k'\} \cup \{k + k'\}) \\
      {} + \sum_{k' = 1}^{k - 1} M^r(\alpha \cup \{k', k - k'\}) \\
      {} + \sum_{k' = 1}^{k - 1} \sum_{r' = 0}^r \sum_{\alpha' \subseteq \alpha} M^{r'}(\alpha' \cup \{k'\}) M^{r - r'}(\alpha \setminus \alpha' \cup \{k - k'\})
  \end{multline}
  for $\alpha \vdash d$, $d, r \geq 0$, $k \geq 1$. In this recurrence, $m_{k'}(\alpha)$ is the number of parts of $\alpha$ of size $k'$, and the last sum is over the $2^{\ell(\alpha)}$ subpartitions $\alpha'$ of $\alpha$.
\end{theorem}

\begin{proof}
  As long as the initial condition and the recurrence relation hold, uniqueness follows by induction on $r$. The initial condition follows from the fact that for $r = 0$ we must have $\sigma = \id$, and the identity permutation is only transitive in $S_1$.
  
  To show the recurrence, fix a permutation $\sigma \in S_d$ of cycle type $\alpha \cup \{k\}$, where the element $d$ is in a cycle of length $k$, and consider a transitive monotone factorization
  \begin{equation}\label{eq:fact1}
    (a_1 \, b_1) (a_2 \, b_2) \cdots (a_r \, b_r) (a_{r+1} \, b_{r+1}) = \sigma.
  \end{equation}
  The transitivity condition forces the element $d$ to appear in some transposition, and the monotonicity condition forces it to appear in the last transposition, so it must be that $b_{r+1} = d$. If we move this transposition to the other side of the equation and set $\sigma' = \sigma (a_{r+1} \, b_{r+1})$, we get the shorter monotone factorization
  \begin{equation}\label{eq:fact2}
    (a_1 \, b_1) (a_2 \, b_2) \cdots (a_r \, b_r) = \sigma'.
  \end{equation}
  Depending on whether $a_{r+1}$ is in the same cycle of $\sigma'$ as $b_{r+1}$ and whether \eqref{eq:fact2} is still transitive, the shorter factorization falls into exactly one of the following three cases, corresponding to the three terms on the right-hand side of the recurrence.
  \begin{description}
    \item[Cut]
      Suppose $a_{r+1}$ and $b_{r+1}$ are in the same cycle of $\sigma'$. Then, $\sigma$ is obtained from $\sigma'$ by cutting the cycle containing $a_{r+1}$ and $b_{r+1}$ in two parts, one containing $a_{r+1}$ and the other containing $b_{r+1}$, so $(a_{r+1} \, b_{r+1})$ is called a \emph{cut} for $\sigma'$, and also for the factorization \eqref{eq:fact1}. Conversely, $a_{r+1}$ and $b_{r+1}$ are in different cycles of $\sigma$, and $\sigma'$ is obtained from $\sigma$ by joining these two cycles, so the transposition $(a_{r+1} \, b_{r+1})$ is called a \emph{join} for $\sigma$. Note that in the case of a cut, \eqref{eq:fact2} is transitive if and only if \eqref{eq:fact1} is transitive.
      
      For $k' \geq 1$, there are $k' m_{k'}(\alpha)$ possible choices for $a_{r+1}$ in a cycle of $\sigma$ of length $k'$ other than the one containing $b_{r+1}$. For each of these choices, $(a_{r+1} \, b_{r+1})$ is a cut and $\sigma'$ has cycle type $\alpha \setminus \{k'\} \cup \{k + k'\}$. Thus, the number of transitive monotone factorizations of $\sigma$ where the last factor is a cut is
      \[
        \sum_{k' \geq 1} k' m_{k'}(\alpha) M^r(\alpha \setminus \{k'\} \cup \{k + k'\}),
      \]
      which is the first term in the recurrence.
    
    \item[Redundant join]
      Now suppose that $(a_{r+1} \, b_{r+1})$ is a join for $\sigma'$ and that \eqref{eq:fact2} is transitive. Then, we say that $(a_{r+1} \, b_{r+1})$ is a \emph{redundant join} for \eqref{eq:fact1}.
      
      The transposition $(a_{r+1} \, b_{r+1})$ is a join for $\sigma'$ if and only if it is a cut for $\sigma$, and there are $k - 1$ ways of cutting the $k$-cycle of $\sigma$ containing $b_{r+1}$. Thus, the number of transitive monotone factorizations of $\sigma$ where the last factor is a redundant join is
      \[
        \sum_{k' = 1}^{k - 1} M^r(\alpha \cup \{k', k - k'\}),
      \]
      which is the second term in the recurrence.
    
    \item[Essential join]
      Finally, suppose that $(a_{r+1} \, b_{r+1})$ is a join for $\sigma'$ and that \eqref{eq:fact2} is not transitive. Then, we say that $(a_{r+1} \, b_{r+1})$ is an \emph{essential join} for \eqref{eq:fact1}. In this case, the action of the subgroup $\langle (a_1 \, b_1), \ldots, (a_r \, b_r) \rangle$ must have exactly two orbits on the ground set, one containing $a_{r+1}$ and the other containing $b_{r+1}$. Since transpositions acting on different orbits commute, \eqref{eq:fact2} can be rearranged into a product of two transitive monotone factorizations on these orbits. Conversely, given a transitive monotone factorization for each orbit, this process can be reversed, and the monotonicity condition guarantees uniqueness of the result.
      
      As with redundant joins, there are $k - 1$ choices for $a_{r+1}$ to split the $k$-cycle of $\sigma$ containing $b_{r+1}$. Each of the other cycles of $\sigma$ must be in one of the two orbits, so there are $2^{\ell(\alpha)}$ choices for the orbit containing $a_{r+1}$. Thus, the number of transitive monotone factorizations of $\sigma$ where the last factor is an essential join is
      \[
        \sum_{k' = 1}^{k - 1} \sum_{r' = 0}^r \sum_{\alpha' \subseteq \alpha} M^{r'}(\alpha' \cup \{k'\}) M^{r - r'}(\alpha \setminus \alpha' \cup \{k - k'\}),
      \]
      which is the third term in the recurrence.
      \qedhere
  \end{description}
\end{proof}

%----------------------------------------------------------------
\subsection{Monotone join-cut equation}\label{sec:jcequation}
%----------------------------------------------------------------

Since the numbers $M^r(\alpha)$ are a rescaled version of the monotone single Hurwitz numbers $H^r(\alpha)$, we can rewrite the recurrence relation for $M^r(\alpha)$ from \autoref{thm:recur} as a partial differential equation for the generating function $\Mon$. The result is the monotone join-cut equation of \autoref{thm:joincut}, which we restate here for convenience.

\restatejoincut*

\begin{proof}
  This equation can be obtained by multiplying the recurrence relation~\eqref{eq:recurrence} by the weight
  \[
    \frac{\abs{\C_\alpha} z^{d+k} t^r p_\alpha p_k}{2\,d\,!}
  \]
  and summing over all choices of $d, \alpha, k, r$ with $d \geq 0$, $\alpha \vdash d$, $k \geq 1$, and $r \geq 0$. The resulting sum can then be rewritten in terms of the generating function $\Mon$ via the defining equations~\eqref{eq:defM} and~\eqref{eq:defMon}, together with the fact that
  \[
    \abs{\C_\alpha} = \frac{d\,!}{\prod_{j \geq 1} j^{m_j(\alpha)} \, m_j(\alpha)!}.
  \]
  This shows that $\Mon$ is indeed a solution of the partial differential equation. To see that the solution is unique, note that apart from $d = 0$, comparing the coefficient of $z^d t^{-1}$ of each side of the partial differential equation uniquely determines $[z^d t^0] \Mon$, and comparing the coefficient of $z^d t^r$ of each side for $r \geq 0$ uniquely determines $[z^d t^{r+1}] \Mon$ in terms of $[z^d t^r] \Mon$.
\end{proof}

%----------------------------------------------------------------
\subsection{Topological recursion}\label{sec:topological}
%----------------------------------------------------------------

In this section, we define a different type of generating function for monotone Hurwitz numbers which is similar to the type of generating function for classical Hurwitz numbers that has previously arisen in the physics literature. Specifically, by analogy with the generating function $H_g(x_1, x_2, \ldots, x_\ell)$ for Hurwitz numbers considered by Bouchard and Mari\~no~\cite[Equations~(2.11) and~(2.12)]{bouchard-marino08}, consider the generating function
\begin{equation}\label{eq:deftopological}
  \mathbf{M}_g(x_1, x_2, \ldots, x_\ell) = \sum_{\alpha_1, \alpha_2, \ldots, \alpha_\ell \geq 1} \frac{\mon_g(\alpha)}{\abs{\C_\alpha}} x_1^{\alpha_1 - 1} x_2^{\alpha_2 - 1} \cdots x_\ell^{\alpha_\ell - 1},
\end{equation}
where we take $\alpha = (\alpha_1, \alpha_2, \ldots, \alpha_\ell)$ to be a \emph{composition}, that is, an $\ell$-tuple of positive integers. One form of recurrence for Hurwitz numbers, expressed in terms of the series $H_g(x_1, x_2, \ldots, x_\ell)$, is referred to as \emph{topological recursion} (see, \textit{e.g.}, \cite[Conjecture~2.1]{bouchard-marino08}; \cite[Remark~4.9]{eynard-mulase-safnuk09}; \cite[Definition~4.2]{eynard-orantin07}). The corresponding recurrence for monotone Hurwitz numbers, expressed in terms of the series $\mathbf{M}_g(x_1, x_2, \ldots, x_\ell)$, is given in the following result.

\begin{theorem}\label{thm:toprec}
  For $g \geq 0$ and $\ell \geq 1$, we have
  \begin{multline}
    \mathbf{M}_g(x_1, x_2, \ldots, x_\ell)
      = \delta_{g,0} \delta_{\ell,1}
      + x_1 \mathbf{M}_{g-1}(x_1, x_1, x_2, \ldots, x_\ell) \vphantom{\sum_{j=2}^\ell} \label{eq:toporec}\\
      + \sum_{j=2}^\ell \diff{x_j} \left( \frac{x_1 \mathbf{M}_g(x_1, \ldots, \widehat{x_j}, \ldots x_\ell) - x_j \mathbf{M}_g(x_2, \ldots, x_\ell)}{x_1 - x_j} \right) \\
      + \sum_{g'=0}^g \sum_{S \subseteq \{2, \ldots, k\}} x_1 \mathbf{M}_{g'}(x_1, x_S) \mathbf{M}_{g-g'}(x_1, x_{\overline{S}}),
  \end{multline}
  where $x_1, \ldots, \widehat{x_j}, \ldots x_\ell$ is the list of all variables $x_1, \ldots, x_\ell$ except $x_j$, $x_S$ is the product of all variables $x_j$ with $j \in S$, and $\overline{S} = \{2, \ldots, k\} \setminus S$.
\end{theorem}

\begin{proof}
  Like the monotone join-cut equation of \autoref{thm:joincut}, this equation can be obtained by multiplying the recurrence~\eqref{eq:recurrence} by a suitable weight and summing over an appropriate set of choices. In this case, the appropriate weight is
  \[
    x_1^{k-1} x_2^{\alpha_1-1} x_3^{\alpha_2-1} \cdots x_\ell^{\alpha_{\ell-1}-1},
  \]
  and the sum is over all positive integer choices of $k, \alpha_1, \alpha_2, \ldots, \alpha_{\ell-1}$. In view of the Riemann-Hurwitz formula~\eqref{eq:riemann-hurwitz}, the appropriate choice of $r$ is
  \[
    r = k + \ell + 2g - 3 + \sum_{j=1}^{\ell-1} \alpha_j.
  \]
  The resulting summations can then be rewritten in terms of the appropriate generating functions by using the defining equations~\eqref{eq:defM} and~\eqref{eq:deftopological}.
\end{proof}

\begin{remark}
  Note that there is an asymmetry between the variable $x_1$ and the variables $x_2, \ldots, x_\ell$ in~\eqref{eq:toporec}, even though $ \mathbf{M}_g(x_1, x_2, \ldots, x_\ell)$ itself is symmetric in all variables.
\end{remark}

For small values of $g$ and $\ell$, the recurrence~ \eqref{eq:toporec} can be solved directly. In particular, we obtain
\begin{align*}
  \mathbf{M}_0(x_1) &= \frac{1 - \sqrt{1 - 4x_1}}{2x_1}, \\
  \mathbf{M}_0(x_1, x_2) &= \frac{4}{\sqrt{1 - 4x_1} \sqrt{1 - 4x_2} (\sqrt{1 - 4x_1} + \sqrt{1 - 4x_2})^2}.
\end{align*}
If we define $y_i$ by $y_i = 1 + x_i y_i^2$ for $i \geq 1$, then these can be rewritten as
\begin{align*}
  \mathbf{M}_0(x_1) &= y_1, \\
  \mathbf{M}_0(x_1, x_2) &= \frac{x_1 \diff[y_1]{x_1} x_2 \diff[y_2]{x_2} (x_2 y_2 - x_1 y_1)^2}{(y_1 - 1) (y_2 - 1) (x_2 - x_1)^2}.
\end{align*}
In the terminology of Eynard and Orantin~\cite{eynard-orantin07}, this seems to mean that we have the spectral curve $y = 1 + xy^2$, but it is unclear to us what the correct notion of Bergmann kernel should be in our case.

%----------------------------------------------------------------
\section{Intermediate forms}\label{sec:intermed}
%----------------------------------------------------------------

In this section, we introduce some algebraic methodology that will allow us to solve the monotone join-cut equation. This  methodology consists of a set of generating functions for monotone Hurwitz numbers of fixed genus, together with families of operators. These allow us to transform the monotone join-cut equation into an algebraic operator equation for these genus-specific generating functions.

%----------------------------------------------------------------
\subsection{Algebraic methodology}\label{sec:liftprojsplit}
%----------------------------------------------------------------

As the first part of our algebraic methodology, we define three families of operators which use a new countable set of indeterminates $\xx = (x_1, x_2, \ldots)$, algebraically independent of $\pp = (p_1, p_2, \ldots)$. We begin with lifting operators.

\begin{definition}
  Let $\xx = (x_1, x_2, \ldots)$ and $\pp = (p_1, p_2, \ldots)$ be countable sets of indeterminates. The $i$th \emph{lifting operator} $\Delta_i$ is the $\QQ[[\xx]]$-linear differential operator on the ring $\QQ[[\xx, \pp]]$ defined by
  \[
    \Delta_i = \sum_{k \geq 1} k x_i^k \diff{p_k}, \qquad i \geq 1.
  \]
\end{definition}

The combinatorial effect of $\Delta_i$, when applied to a generating function, is to pick a cycle marked by $p_k$ in all possible ways and mark it by $k x_i^k$ instead, that is, by $x_i^k$ once for each element of the cycle. Note that $\Delta_i x_j = 0$ for all $j$, so that
\[
  \Delta_i^2 = \sum_{j,k \geq 1} jk \, x_i^{j+k} \sdiff{p_j}{p_k}.
\]

Accompanying these lifting operators, we also have projection operators.

\begin{definition}
  Let $\xx = (x_1, x_2, \ldots)$ and $\pp = (p_1, p_2, \ldots)$ be countable sets of indeterminates. The $i$th \emph{projection operator} $\Pi_i$ is the $\QQ[[\pp]]$-linear idempotent operator on the ring $\QQ[[\xx, \pp]]$ defined by
  \[
    \Pi_i = [x_i^0] + \sum_{k \geq 1} p_k [x_i^k], \qquad i \geq 1.
  \]
\end{definition}

The combinatorial effect of $\Pi_i$, when applied to a generating function, is to take any cycle marked by $x_i^k$ and mark it by $p_k$ instead. The combined effect of a lift and a projection when applied to a generating function in $\QQ[[\pp]]$ is given by
\[
  \Pi_i \Delta_i = \sum_{k \geq 1} k p_k \diff{p_k}.
\]

Finally, we introduce splitting operators.

\begin{definition}
  Let $F(x_i)$ be an element of $\QQ[[\xx, \pp]]$, considered as a power series in $x_i$, and let $j \geq 1$ be an index other than $i \geq 1$. Then the $i$-to-$j$ \emph{splitting} operator is defined by
  \[
    \Split_{i \to j} F(x_i) = \frac{x_j F(x_i) - x_i F(x_j)}{x_i - x_j} + F(0),
  \]
  so that
  \[
    \Split_{i \to j} x_i^k = x_i^{k-1} x_j + x_i^{k-2} x_j^2 + \cdots + x_i x_j^{k-1}.
  \]
\end{definition}

Combinatorially, the effect of $\Split_{i \to j}$ on a generating function is to take the cycle marked by $x_i$ and split it in two cycles, marked by $x_i$ and $x_j$ respectively, in all possible ways. The combined effect of a lift, a split and a projection on a generating function in $\QQ[[\pp]]$ is
\[
  \Pi_1 \Pi_2 \Split_{1 \to 2} \Delta_1 = \sum_{i,j \geq 1} (i + j) p_i p_j \diff{p_{i + j}}.
\]

As the second part of our algebraic methodology, we define the generating functions
\begin{equation}\label{eq:defMon_g}
  \Mon_g = \sum_{d \geq 1} \sum_{\alpha \vdash d} \frac{\mon_g(\alpha) p_\alpha}{d\,!}, \qquad g \geq 0.
\end{equation}
Thus, $\Mon_g$ is the generating function for genus $g$ monotone single Hurwitz numbers where, combinatorially, $p_1, p_2, \ldots$ are ordinary markers for the parts of $\alpha$, and there is an implicit exponential marker for the size $d$ of the ground set.

%----------------------------------------------------------------
\subsection{Recasting the monotone join-cut equation}\label{sec:recasting}
%----------------------------------------------------------------

We are now able to recast monotone join-cut equation as an algebraic operator equation involving our genus-specific generating series $\Mon_g$ and the three operators that we have introduced above.

\begin{theorem}\label{thm:pdeunique}
  \ \par % NOTE: without something like this, references to "thm:pdeunique" do not work correctly.
  \begin{enumerate}[(i)]
    \item
      The generating function $\Delta_1 \Mon_0$ is the unique formal power series solution of the partial differential equation
      \begin{equation}\label{eq:operator-genus-zero}
        \Delta_1 \Mon_0 = \Pi_2 \Split_{1 \to 2} \Delta_1 \Mon_0 + (\Delta_1 \Mon_0)^2 + x_1
      \end{equation}
      with the initial condition $[p_0 x_1^0] \Delta_1 \Mon_0 = 0$.
    
    \item
      For $g \geq 1$, the generating function $\Delta_1 \Mon_g$ is uniquely determined in terms of $\Delta_1 \Mon_0, \Delta_1 \Mon_1, \ldots, \Delta_1 \Mon_{g-1}$ by the equation
      \begin{equation}\label{eq:operator-higher-genus}
        \left( 1 - 2 \Delta_1 \Mon_0 - \Pi_2 \Split_{1 \to 2} \right) \Delta_1 \Mon_g = \Delta_1^2 \Mon_{g-1} + \sum_{g'=1}^{g-1} \Delta_1 \Mon_{g'} \, \Delta_1 \Mon_{g-g'}.
      \end{equation}
    
    \item
      For $g \geq 0$, the generating function $\Mon_g$ is uniquely determined by the generating function $\Delta_1 \Mon_g$ and the fact that $[p_0] \Mon_g = 0$.
  \end{enumerate}
\end{theorem}

\begin{proof}
  As with \autoref{thm:joincut} and \autoref{thm:toprec}, this result is obtained by multiplying the recurrence from \autoref{thm:recur} by a suitable weight and summing over a set of possible choices. Consider the generating function
  \[
    F = \sum_{g \geq 0} u^g \Mon_g = \sum_{g \geq 0} u^g \sum_{d \geq 1} \sum_{\alpha \vdash d} \frac{\mon_g(\alpha) p_\alpha}{d\,!},
  \]
  where $u$ is an ordinary for the genus $g$. In view of~\eqref{eq:riemann-hurwitz} and~\eqref{eq:defM}, we have
  \begin{align*}
    F &= \sum_{g \geq 0} u^g \sum_{d \geq 1} \sum_{\alpha \vdash d} \frac{M^{d+\ell(\alpha)+2g-2}(\alpha) \, p_\alpha}{\prod_{j \geq 1} j^{m_j(\alpha)} \, m_j(\alpha)!}, \\
    \Delta_1 F &= \sum_{g \geq 0} u^g \sum_{d \geq 0} \sum_{\alpha \vdash d} \sum_{k \geq 1} \frac{M^{d+k+\ell(\alpha)+2g-1}(\alpha \cup \{k\}) \, p_\alpha x_1^k}{\prod_{j \geq 1} j^{m_j(\alpha)} \, m_j(\alpha)!}.
  \end{align*}
  Thus, by multiplying the recurrence~\eqref{eq:recurrence} by the weight
  \[
    \frac{u^g p_\alpha x_1^k}{\prod_{j \geq 1} j^{m_j(\alpha)} \, m_j(\alpha)!}
  \]
  and summing over all choices of $g, d, \alpha, k$ with $g \geq 0$, $d \geq 0$, $\alpha \vdash d$, $k \geq 1$, and $r = d+k+\ell(\alpha)+2g-2$, we obtain the partial differential equation
  \begin{equation}\label{eq:auxdiffeq}
    \Delta_1 F - x_1 = \Pi_2 \Split_{1 \to 2} \Delta_1 F + u \Delta_1^2 F + (\Delta_1 F)^2.
  \end{equation}
  To show that $\Delta_1 F$ is the unique solution of this partial differential equation with $[u^0 p_0 x_1^0] \Delta_1 F = 0$, note that~\eqref{eq:auxdiffeq} exactly captures the recurrence of \autoref{thm:recur}, so each non-constant coefficient of $\Delta_1 F$ is uniquely determined.
  
  Extracting the coefficient of $u^g$ from~\eqref{eq:auxdiffeq} and rearranging terms gives the monotone join-cut equations of the theorem statement for $g = 0$ and $g \geq 1$.
  
  Finally, note that given $\Delta_1 \Mon_g$, we can compute
  \[
    \Pi_1 \Delta_1 \Mon_g = \sum_{d \geq 1} \sum_{\alpha \vdash d} \frac{\mon_g(\alpha) p_\alpha}{(d-1)!},
  \]
  which uniquely determines every coefficient of $\Mon_g$ except for the constant term.
\end{proof}

\begin{remark}
  This form of the monotone join-cut equation is technically slightly stronger than the one given in \autoref{thm:joincut}, since it is obtained from the recurrence relation in \autoref{thm:recur} by using a less symmetric weight.
\end{remark}

\begin{remark}
  The monotone join-cut equation for higher genera will be the subject of a forthcoming paper~\cite{goulden-guay-paquet-novak12b}, which we won't discuss further here.
\end{remark}

%----------------------------------------------------------------
\section{Transformed  variables and proof of the main result}\label{sec:solution}
%----------------------------------------------------------------

In this section, we prove the main result, \autoref{thm:gzformula}. Our strategy is to define the series $\guess$ by
\begin{equation}\label{eq:Fseries}
  \guess = \sum_{d \geq 1} \sum_{\alpha \vdash d} \frac{p_\alpha}{\abs{\Aut \alpha}} (2d + 1)^{\overline{\ell(\alpha) - 3}} \prod_{j=1}^{\ell(\alpha)} \binom{2\alpha_j}{\alpha_j},
\end{equation}
and then to show that the series $\Delta_1 \guess$ satisfies the genus zero monotone join-cut equation~\eqref{eq:operator-genus-zero}.

\begin{remark}
  We initially conjectured this formula for genus zero monotone Hurwitz numbers after generating extensive numerical data, using the group algebra approach described in \autoref{sec:group-algebra}, together with the character theory and generating series capabilities of Sage~\cite{sage}. In particular, the case where $\alpha$ has $\ell(\alpha) = 3$ parts was very suggestive, since the formula then breaks down into a product of three terms. This was also our first indication of the striking similarities between monotone Hurwitz numbers and classical Hurwitz numbers.
\end{remark}

%----------------------------------------------------------------
\subsection{Transformed variables and Lagrange inversion}\label{sec:lagrange}
%----------------------------------------------------------------

In working with the series $\guess$, it is convenient to change variables from $\pp = (p_1, p_2, \ldots)$ to $\qq = (q_1, q_2, \ldots)$, where
\begin{equation}\label{eq:pqrel}
  q_j = p_j \left(1 - \sum_{k \geq 1} \binom{2k}{k} q_k\right)^{-2j}, \qquad j \geq 1.
\end{equation}
This change of variables is invertible, and can be carried out using the Lagrange Implicit Function Theorem in many variables (see~\cite{goulden-jackson04}). The first result expressing $\guess$ in terms of the new indeterminates $\qq$ involves the differential operator
\begin{equation}\label{eq:defDD}
  \DD = \sum_{k \geq 1} k p_k\diff{p_k}.
\end{equation}

\begin{theorem}\label{thm:F3D}
Let $\gamma = \sum_{k \geq 1} \binom{2k}{k} q_k$ and $\eta = \sum_{k \geq 1} (2k + 1) \binom{2k}{k} q_k$. Then
\[
  (2\DD - 2) (2\DD - 1) (2\DD) \guess = \frac{(1 - \gamma)^3}{1 - \eta} - 1.
\]
\end{theorem}

\begin{proof}
  From~\eqref{eq:Fseries}, for any $\alpha \vdash d$ with $d \geq 1$, we have
  \begin{align*}
    [p_\alpha] (2\DD - 2) (2\DD - 1) (2\DD) \guess
      &= \frac{1}{\abs{\Aut \alpha}} (2d - 2)^{\overline{\ell(\alpha)}} \prod_{j=1}^{\ell(\alpha)} \binom{2\alpha_j}{\alpha_j} \\
      &= \frac{(-1)^\ell \ell(\alpha)!}{\abs{\Aut \alpha}} \binom{2 - 2d}{\ell(\alpha)} \prod_{j=1}^{\ell(\alpha)} \binom{2\alpha_j}{\alpha_j},
  \end{align*}
  and we conclude that
  \begin{equation}\label{eq:DFgamma}
    [p_\alpha] (2\DD - 2) (2\DD - 1) (2\DD) \guess =  [q_\alpha] (1 - \gamma)^{2 - 2d}.
  \end{equation}
  Now let $\phi_j = (1 - \gamma)^{-2j}$, so that~\eqref{eq:pqrel} becomes $q_j = p_j \phi_j$, $ j\geq 1$. Then, from the multivariate Lagrange Implicit Function Theorem~\cite[Theorem 1.2.9]{goulden-jackson04}, for any formal power series $\Phi \in \QQ[[\qq]]$, we obtain
  \begin{align*}
    [p_\alpha] \Phi
      &= [q_\alpha] \Phi \, \phi_\alpha \det\left( \delta_{ij} - q_j \diff{q_j} \log \phi_i \right)_{i,j \geq 1} \\
      &= [q_\alpha] \Phi \, \phi_\alpha \det\left( \delta_{ij} - \frac{2i q_j}{1 - \gamma} \binom{2j}{j} \right)_{i,j \geq 1},
  \end{align*}
  where $\phi_\alpha = \prod_{j \geq 1} \phi_{\alpha_j}$. Then we have $\phi_\alpha = (1 - \gamma)^{-2d}$, and using the fact that $\det(I + M) = 1 + \trace(M)$ for any matrix $M$ of rank zero or one, we can evaluate the determinant as
  \[
    \det\left( \delta_{ij} - q_j \diff{q_j} \log \phi_i \right)_{i,j \geq 1}
      = 1 - \sum_{k \geq 1} \frac{2k q_k}{1 - \gamma} \binom{2k}{k}
      = \frac{1 - \eta}{1 - \gamma}.
  \]
  Substituting, we obtain
  \[
    [p_\alpha] \Phi = [q_\alpha] \frac{(1 - \eta) \Phi}{(1 - \gamma)^{2d + 1}}.
  \]
  Comparing this result with~\eqref{eq:DFgamma}, we obtain
  \[
    [p_\alpha] (2\DD - 2) (2\DD - 1) (2\DD) \guess = [p_\alpha] \frac{(1 - \gamma)^3}{1 - \eta}
  \]
  for $\alpha \vdash d$ and $d \geq 1$, and computing the constant term separately, the result follows immediately.
\end{proof}

%----------------------------------------------------------------
\subsection{Inverting differential operators in the transformed variables}\label{sec:invert}
%----------------------------------------------------------------

In order to use \autoref{thm:F3D} to evaluate $\Delta_1 \guess$, we need to invert the differential operators $2\DD - 2$, $2\DD - 1$, and $2\DD$. We will work with the transformed variables $\qq$, and thus introduce the additional differential operators
\[
  \D_k = p_k\diff{p_k}, \qquad \E_k = q_k\diff{q_k}, \qquad \EE = \sum_{k \geq 1} k q_k\diff{q_k}.
\]
As $\QQ$-linear operators, the operators $\D_1, \D_2, \ldots$ and $\DD$ have an eigenbasis given by the set $\{p_\alpha \colon \alpha \vdash d,\, d \geq 0\}$, and consequently they commute with each other. Similarly, the operators $\E_1, \E_2, \ldots$ and $\EE$ have the set $\{q_\alpha \colon \alpha \vdash d,\, d \geq 0\}$ as an eigenbasis and commute with each other. However, these two families of operators don't commute with each other. By using the relation~\eqref{eq:pqrel} to compute the action of $\E_k$ on $p_j$, we can verify the operator identity
\[
  \E_k = \D_k - \frac{2q_k}{1 - \gamma} \binom{2k}{k} \DD, \qquad k \geq 1.
\]
It follows that
\[
  \EE = \frac{1 - \eta}{1 - \gamma} \DD,
\]
and we can deduce the identity
\begin{equation}\label{eq:DtoE}
  \D_k = \E_k + \frac{2q_k}{1 - \eta} \binom{2k}{k} \EE, \qquad k \geq 1.
\end{equation}
Thus, we can express these differential operators for $\pp$ and $\qq$ in terms of each other. In the following result, we apply these expressions to invert the differential operators that appear in \autoref{thm:F3D}.

\begin{theorem}\label{thm:intF}
For $k \geq 1$, we have
\[
  \D_k \guess = \frac{1}{2k(2k - 1)} \binom{2k}{k} q_k - \sum_{j \geq 1} \frac{2j + 1}{2(j + k)(2k - 1)} \binom{2j}{j} \binom{2k}{k} q_j q_k.
\]
\end{theorem}

\begin{proof}
  As notation local to this proof, let
  \[
    \guess''' = (2\DD - 2) (2\DD - 1) (2\DD) \guess, \quad
    \guess''  =            (2\DD - 1) (2\DD) \guess, \quad
    \guess'   =                       (2\DD) \guess.
  \]
  To prove the result, we use the operator identity
  \begin{equation}\label{eq:intDD}
      (1 - \gamma)^i (2\EE - i) \Big( (1 - \gamma)^{-i} G \Big) = \frac{1 - \eta}{1 - \gamma} (2\DD - i) (G),
  \end{equation}
  which holds for any integer $i$ and any formal power series $G$. This allows us to express the differential operators $2\DD - 2$, $2\DD - 1$, and $2\DD$ in terms of the operators $2\EE - 2$, $2\EE - 1$, and $2\EE$, which we can invert by recalling that they have $\{q_\alpha \colon \alpha \vdash d,\, d \geq 0\}$ as an eigenbasis.
  
  We proceed in a number of stages. First we invert $2\DD - 2$ by applying~\eqref{eq:intDD} with $i = 2$ to \autoref{thm:F3D}, obtaining
  \begin{align*}
    \guess''
      &= (2\DD - 2)^{-1} (\guess''') \\
      &= \frac{1}{2} + (2\DD - 2)^{-1} \left( \frac{(1 - \gamma)^3}{1 - \eta} \right) \\
      &= \tfrac{1}{2} + (1 - \gamma)^2 (2\EE - 2)^{-1} (1) \\
      &= \tfrac{1}{2} - \tfrac{1}{2} (1 - \gamma)^2,
  \end{align*}
  after checking separately that $[p_1] \guess'' = 2$. (We need to check this since the kernel of $2\DD - 2$ is spanned by $p_1$.)
  
  Next we apply $\D_k$ to $\guess''$ via~\eqref{eq:DtoE}. This is straightforward, and gives
  \[
    \D_k \guess'' = \frac{(1 - \gamma)^2}{1 - \eta} \binom{2k}{k} q_k.
  \]
  
  Now we invert $2\DD - 1$ by applying~\eqref{eq:intDD} with $i = 1$, which gives
  \begin{align*}
    \D_k \guess'
      &= (2\DD - 1)^{-1} (\D_k \guess'') \\
      &= (1 - \gamma) (2\EE - 1)^{-1} \left( \binom{2k}{k} q_k \right) \\
      &= (1 - \gamma) \frac{1}{2k - 1} \binom{2k}{k} q_k.
  \end{align*}
  
  Finally, we invert $2\DD$ by applying~\eqref{eq:intDD} with $i = 0$, giving
  \begin{align*}
    \D_k \guess
      &= (2\DD)^{-1} (\D_k \guess') \\
      &= (2\EE)^{-1} \left( (1 - \eta) \frac{1}{2k - 1} \binom{2k}{k} q_k \right) \\
      &= (2\EE)^{-1} \left( \frac{1}{2k - 1} \binom{2k}{k} q_k - \sum_{j \geq 1} \frac{2j + 1}{2k - 1} \binom{2j}{j} \binom{2k}{k} q_j q_k \right) \\
      &= \frac{1}{2k(2k - 1)} \binom{2k}{k} q_k - \sum_{j \geq 1} \frac{2j + 1}{2(j + k)(2k - 1)} \binom{2j}{j} \binom{2k}{k} q_j q_k.
  \end{align*}
  Again, the constant term needs to be checked separately, since the kernel of $(2\DD)$ consists of the constants, but clearly $\D_k \guess$ has no constant term.
\end{proof}

%----------------------------------------------------------------
\subsection{The generating function for genus zero}\label{sec:genuszero}
%----------------------------------------------------------------

In order to work consistently in the tranformed variables $\qq$, it will be useful to have descriptions of the projection and splitting operators in terms of $\qq$. When considering these operators, the change of variables from $\pp$ to $\qq$ also corresponds to a change of variables from $\xx$ to a new countable set of indeterminates $\yy = (y_1, y_2, \ldots)$, where we impose the relations
\begin{equation}\label{eq:xyrel}
  y_i = x_i (1 - \gamma)^{-2}.
\end{equation}
We can express the indeterminates $\pp$ and $\qq$ in terms of each other using~\eqref{eq:pqrel}, so we can identify the rings $\QQ[[\pp]]$ and $\QQ[[\qq]]$. Since $(1 - \gamma)^{-2}$ is an invertible element in this ring, we can further identify the rings $\QQ[[\pp, \xx]]$ and $\QQ[[\qq, \yy]]$ using~\eqref{eq:xyrel}. In this bigger ring, we have the operator identities
\begin{align*}
  \Pi_i &= [x_i^0] + \sum_{k \geq 1} p_k [x_i^k] = [y_i^0] + \sum_{k \geq 1} q_k [y_i^k], \\
  \Split_{i \to j} G(x_i) &= \frac{x_j G(x_i) - x_i G(x_j)}{x_i - x_j} + G(0) = \frac{y_j G(x_i) - y_i G(x_j)}{y_i - y_j} + G(0),
\end{align*}
so the projection and splitting operators are just as easy to use with either set of indeterminates.

\begin{remark}
  For completeness, note that the lifting operators can also be described in terms of $\qq$ and $\yy$, although the expressions are somewhat more complicated. That is, using~\eqref{eq:DtoE}, we obtain the expression
  \[
    \Delta_i = \sum_{k \geq 1} k x_i^k \diff{p_k} = \sum_{k \geq 1} \left( k y_i^k \diff{q_k} \right) + \frac{4y_i (1 - 4y_i)^{-\frac{3}{2}}}{(1 - \eta)} \sum_{k \geq 1} \left( k q_k \diff{q_k} + y_k \diff{y_k} \right).
  \]
\end{remark}

We are now able to evaluate $\Delta_1 \guess$ in the indeterminates $\qq$ and $\yy$.

\begin{corollary}\label{thm:Del_1F}
  We have
  \[
    \Delta_1 \guess = \Pi_2 \left( 1 - \sqrt{1 - 4y_1} - \frac{y_1}{2(y_1 - y_2)} \left( 1 - \sqrt{\frac{1 - 4y_1}{1 - 4y_2}} \right) \right).
  \]
\end{corollary}

\begin{proof}
  From \autoref{thm:intF}, we have
  \begin{align*}
    \Delta_1 \guess
      &= \sum_{k \geq 1} \frac{k y_1^k}{q_k} \D_k \guess \\
      &= \sum_{k \geq 1} \frac{1}{2(2k - 1)} \binom{2k}{k} y_1^k - \sum_{j,k \geq 1} \frac{(2j + 1) k}{2(j + k)(2k - 1)} \binom{2j}{j} \binom{2k}{k} y_1^k q_j \\
      &= \Pi_2\left( 2 G(y_1, 0) - G(y_1, y_2) \right),
  \end{align*}
  where the power series $G(y_1, y_2)$ is defined by
  \[
    G(y_1, y_2) = \sum_{j \geq 0} \sum_{k \geq 1} \frac{(2j + 1)k}{2(j + k)(2k - 1)} \binom{2j}{j} \binom{2k}{k} y_1^k y_2^j.
  \]
  Then, the computation
  \begin{align*}
    G(y_1, y_2)
      &= \int\limits_0^1 y_1 t (1 - 4y_1 t)^{-\frac{1}{2}} (1 - 4y_2 t)^{-\frac{3}{2}} \frac{\mathrm{d}t}{t} \\
      &= \left[ \frac{-y_1}{2(y_1 - y_2)} (1 - 4y_1 t)^{\frac{1}{2}} (1 - 4y_2 t)^{-\frac{1}{2}} \right]_{t = 0}^1 \\
      &= \frac{y_1}{2(y_1 - y_2)} \left( 1 - \sqrt{\frac{1 - 4y_1}{1 - 4y_2}} \right) \\
  \end{align*}
  completes the proof.
\end{proof}

In the following result, using the above explicit expression for $\Delta_1 \guess$, we uniquely identify $\guess$ as the generating function for monotone single Hurwitz numbers in genus zero.

\begin{theorem}\label{thm:idenguess}
  The series $\guess$ satisfies the genus zero monotone join-cut equation
  \[
    \Delta_1 \guess = \Pi_2 \Split_{1 \to 2} \Delta_1 \guess + (\Delta_1 \guess)^2 + x_1.
  \]
  Thus, $\guess = \Mon_0$ is the generating function for monotone single Hurwitz numbers in genus zero.
\end{theorem}

\begin{proof}
\autoref{thm:Del_1F} gives $\Delta_1 \guess = \Pi_2 A(y_1,y_2)$, where
  \[
    A(y_1, y_2) = 1 - \sqrt{1 - 4y_1} - \frac{y_1}{2(y_1 - y_2)} \left( 1 - \sqrt{\frac{1 - 4y_1}{1 - 4y_2}} \right).
  \]
  We need to check that the expression
  \[
    \Delta_1 \guess - \Pi_2 \Split_{1 \to 2} \Delta_1 \guess - (\Delta_1 \guess)^2 - x_1
  \]
  is zero. To do so, we rewrite each of the terms in this expression as
  \begin{align*}
    \Delta_1 \guess &= \Pi_2 \Pi_3 \Big( A(y_1, y_2) \Big), \\
    \Pi_2 \Split_{1 \to 2} \Delta_1 \guess &= \Pi_2 \Pi_3 \left( \frac{y_2 A(y_1, y_3) - y_1 A(y_2, y_3)}{y_1 - y_2} \right), \\
    (\Delta_1 \guess)^2 &= \Pi_2 \Pi_3 \Big( A(y_1, y_2) A(y_1, y_3) \Big), \\
    x_1 = y_1 (1 - \gamma)^2 &= \Pi_2 \Pi_3 \left( y_1 \left( 2 - \frac{1}{\sqrt{1 - 4y_2}} \right) \left( 2 - \frac{1}{\sqrt{1 - 4y_3}} \right) \right)
  \end{align*}
  to get an expression of the form
  \[
    \Pi_2 \Pi_3 \, B(y_1, y_2, y_3).
  \]
  The series $B(y_1, y_2, y_3)$ itself is not zero, but a straightforward computation shows that the series
  \[
    \tfrac{1}{2} B(y_1, y_2, y_3) + \tfrac{1}{2} B(y_1, y_3, y_2),
  \]
  obtained by symmetrizing with respect to $y_2$ and $y_3$, is zero. Thus we have
  \[
    \Pi_2 \Pi_3 \, B(y_1, y_2, y_3) = \Pi_2 \Pi_3 \big( \tfrac{1}{2} B(y_1, y_2, y_3) + \tfrac{1}{2} B(y_1, y_3, y_2) \big) = 0,
  \]
 which completes the verification. The fact that $\guess = \Mon_0$ follows immediately from \autoref{thm:pdeunique}.
\end{proof}

Finally, we are now able to deduce our main result.

\begin{proof}[Proof of \autoref{thm:gzformula}]
  In view of~\eqref{eq:defMon_g} with $g = 0$, applying \autoref{thm:idenguess} gives
  \[
    \mon_0(\alpha) = d\,! [p_{\alpha}] \guess
  \]
  for any partition $\alpha$ of $d \geq 1$. The result follows immediately from~\eqref{eq:Fseries}.
\end{proof}

%----------------------------------------------------------------
\proof[Acknowledgements]
%----------------------------------------------------------------

It is a pleasure to acknowledge helpful conversations with our colleagues Sean Carrell and David Jackson, Waterloo, and Ravi Vakil, Stanford. J.~N.~would like to acknowledge email correspondence with Mike Roth, Queen's. The extensive numerical computations required in this project were performed using Sage~\cite{sage}, and its algebraic combinatorics features developed by the Sage-Combinat community~\cite{sage-combinat}.

\bibliographystyle{myplain}
\bibliography{monotone}
\end{document}